\documentclass[10pt]{amsart}

\usepackage{graphicx}
\usepackage[all]{xy}
\usepackage{amssymb}
\usepackage{latexsym}
\usepackage{amsmath}
\usepackage{mathrsfs}
\usepackage{array}
\usepackage{verbatim}
\usepackage{color}

\usepackage{tikz}
\usetikzlibrary{arrows}

\usepackage{hyperref}

\theoremstyle{plain}
\newtheorem{theorem}{Theorem}

\newtheorem{lemma}[theorem]{Lemma}
\newtheorem{corollary}[theorem]{Corollary}

\newcommand{\Db}{{\rm D}^b}
\newcommand{\Kb}{{\rm K}^b}
\newcommand{\Cb}{{\rm C}^b}
\newcommand{\Kbp}{{\rm K}^{b,p}}

\newcommand{\GL}{\mathsf{GL}}
\renewcommand{\Im}{\mathrm{Im\,}}

\newcommand{\kk}{\mathbf{k}}
\newcommand{\bA}{\mathbf{A}}
\newcommand{\bB}{\mathbf{B}}

\DeclareMathOperator{\Hom}{\mathrm{Hom}}
\DeclareMathOperator{\End}{\mathrm{End}}
\DeclareMathOperator{\Aut}{\mathrm{Aut}}
\DeclareMathOperator{\Out}{\mathrm{Out}}
\DeclareMathOperator{\Inn}{\mathrm{Inn}}
\DeclareMathOperator{\Pic}{\mathrm{Pic}}
\DeclareMathOperator{\TrPic}{\mathrm{TrPic}}

\DeclareMathOperator{\Coker}{\mathrm{Coker}}

\DeclareMathOperator{\add}{\mathrm{add}}
\DeclareMathOperator{\stilt}{\mathrm{silt}}

\DeclareMathOperator{\Filt}{\mathrm{Filt}}

\DeclareMathOperator{\Id}{\mathrm{Id}}

\title[Quasi-hereditary algebras with two simple modules]{On the derived category of quasi-hereditary algebras with two simple modules}

\author{Y. Volkov}

\begin{document}

\begin{abstract}
We describe the derived Picard groups and two-term silting complexes for quasi-hereditary algebras with two simple modules. We also describe by quivers with relations all algebras derived equivalent to a quasi-hereditary algebra with two simple modules.
\end{abstract}

\maketitle

%{\it 2010 Mathematics Subject Classification:} 16E35

%{\it Key words and phrases:} derived Picard group, quasi-hereditary algebra, silting complex, derived equivalence

\section{Introduction}

Derived categories play an important role in many branches of mathematics. Due to \cite{Ric2} if two derived categories of algebras over a field are equivalent, then there is an equivalence induced by a tensor product with a tilting complex of bimodules.
Such equivalences are closed under composition, and thus autoequivalences induced by tilting complexes of bimodules form a subgroup of the group of derived autoequivalences. This group is called {\it the derived Picard group} of an algebra and was first introduced in \cite{RZ} and \cite{Ye1}. Later it was shown in \cite{Ye2} that this group is locally algebraic. It was studied for many classes of algebras. For example, it was described for hereditary algebras in \cite{MY}, for selfinjective Nakayama algebras in \cite{VZ1,VZ2,Zvo} and for preprojective algebras of Dynkin quivers in \cite{AM,Miz}. The question if this group coincides with th group of all autoequivalences or not is still open. 

As we have already mentioned, equivalences of derived categories of algebras are strongly connected with tilting complexes. A generalization of a tilting complex is a silting complex that in some sense gives an equivalence between the derived category of the original algebra and the derived category of some DG algebra.
It was shown recently (see \cite{AIR,Asa,DIJ,IJY}) that two-term silting complexes constitute a very interesting and deep object related to other important notions such as $\tau$-tilting modules and bricks. These structures were studied, for example, for Nakayama algebras in \cite{Ada}, for the Auslander algebra of $\kk[x]/(x^n)$ in \cite{IZ}, for Brauer graph algebras in \cite{AAC} and for preprojective algebras associated with symmetrizable Cartan matrices in \cite{FG}.

Quasi-hereditary algebras are finite dimensional algebras introduced in \cite{CPS2,Sco}. Examples of these algebras arise in the representation theory of groups and Lie algebras. Also they include all finite dimensional algebras of global dimension $2$ (see \cite{DR}).

This paper is devoted to quasi-hereditary algebras with two simple modules. Such algebras were classified in \cite{MH}. Moreover, it was shown in \cite{DUB} that these algebras are exactly finite dimensional algebras of global dimension $1$ and $2$ having two simple modules.

Let us recall some results of \cite{DUB}. The author showed how one can obtain all tilting modules over a finite dimensional algebra of global dimension $2$ with two simple modules. It was also proved that any algebra whose derived category is generated by an exceptional pair is derived equivalent to one of the algebras considered in this paper. Moreover, all algebras derived equivalent to an algebra  of global dimension $2$ with two simple modules were described as endomorphism algebras of some tilting modules described there. All of these algebras are shown to be of global dimension $3$.
We will continue the study initiated in \cite{DUB}.

Using results of \cite{DUB} we will describe the derived Picard groups of all quasi-hereditary algebras with two simple modules. We will also describe all two-term silting complexes over these algebras and using this we will get the description of all algebras derived equivalent to quasi-hereditary algebras with two simple modules.
Our description will be more concrete than the description from \cite{DUB}; namely, we will describe all algebras from our classification by quivers with relations. Note also that due to results of \cite{LY} we obtain a description of all algebras with two simple modules that are not derived simple.

\section{Silting and tilting complexes}

We fix some algebraically closed field $\kk$. All considered vector spaces and algebras are over $\kk$. Whenever we say a module in this paper we mean a right module. For an algebra $\Lambda$, we denote by $\Cb_{\Lambda}$, $\Kb_{\Lambda}$, $\Kbp_{\Lambda}$ and $\Db_{\Lambda}$ the category of bounded complexes of finitely generated $\Lambda$-modules, the bounded homotopy category of finitely generated $\Lambda$-modules, the  bounded homotopy category of  finitely generated projective $\Lambda$-modules and the bounded derived category of  finitely generated $\Lambda$-modules respectively. We denote by $J_\Lambda$ the Jacobson radical of $\Lambda$. All complexes in this paper are equipped with a differential of degree one and for $X\in \Cb_{\Lambda}$ the complex $X[r]$ ($r\in\mathbb{Z}$) has terms $X[r]_i=X_{i+r}$ and differential $(-1)^rd_X$, where $d_X$ is the differential of $X$.

Let us recall that $X\in\Kbp_{\Lambda}$ is called {\it presilting complex} if $\Hom_{\Kb_{\Lambda}}(X,X[i])=0$ for $i>0$. If the same property is satisfied for all $i\not=0$, then $X$ is called {\it pretilting}.
If $X$ is presilting (pretilting) and additionally the smallest full triangulated subcategory of $\Kbp_{\Lambda}$ which contains $X$ and is closed under direct summands coincides with $\Kbp_{\Lambda}$, then $X$ is called {\it silting (tilting)}. The complex $X$ is called {\it basic} if there is no complex $Y$ not isomorphic to zero in $\Kbp_{\Lambda}$ and complex $X'$ such that $X\cong Y^2\oplus X'$ in $\Kbp_{\Lambda}$.

It was proved in \cite{Ric} that if the algebras $\Lambda$ and $\Gamma$ are derived equivalent, then there exists a tilting complex $X\in\Kbp_{\Lambda}$ such that $\End_{\Kb_{\Lambda}}(X)$ is isomorphic to $\Gamma$ as a $\kk$-algebra.
In the same paper it is explained how to construct an equivalence from $\Db_{\Gamma}$ to $\Db_{\Lambda}$ sending $\Gamma$ to $X$ using the tilting complex $X$ and an algebra isomorphism $\Gamma\cong\End_{\Kb_{\Lambda}}(X)$.
One can look also in \cite{VZ1, VZ2} how to construct an equivalence from $\Kbp_{\Gamma}$ to $\Kbp_{\Lambda}$ using the same data. In the current paper we will use the fact that if $U=(U_i\rightarrow\dots\rightarrow U_j)$ is an object of $\Kbp_{\Gamma}$, then the corresponding equivalence sends $U$ to a totalization of a bicomplex whose $k$-th column is the images of $U_k$ under this equivalence, while the image of $U_k$ can be calculated using the fact that $U_k$ is a direct sum of direct summands of $\Gamma$.

Equivalences that can be constructed using the just mentioned algorithm are called standard equivalences. Standard equivalences from $\Kbp_{\Lambda}$ to itself considered modulo natural isomorphisms constitute a group under composition which is called the {\it derived Picard group} of $\Lambda$ and is denoted by $\TrPic(\Lambda)$ (see \cite{VZ1}). This group was first introduced in \cite{RZ,Ye1} as a group of isomorphism classes of two-sided tilting complexes of $\Lambda$-bimodules under the operation of derived tensor product.

Let us recall also  some basic facts about two-term silting complexes. Whenever we say two-term complex in this paper, we mean a complex concentrated in degrees $-1$ and $0$. Since we need to deal only with finite dimensional algebras having two simple modules, we restrict our consideration to this type of algebras.
Then any basic two-term silting complex has two indecomposable presilting direct summands. Let $X=X_0\oplus X_1$ be a two-term silting complex, where $X_0$ and $X_1$ are indecomposable direct summands. Then, for $i=0,1$ one can define complexes $\mu^{\pm}_{X_i}(X_{1-i})$ in the following way. Let $f_1,\dots,f_n\in\Hom_{\Kb_{\Lambda}}(X_{1-i},X_i)$ be such that their classes constitute a basis of
$$\Hom_{\Kb_{\Lambda}}(X_{1-i},X_i)/\big(J_{\End_{\Kb_{\Lambda}(X_{i})}}\Hom_{\Kb_{\Lambda}}(X_{1-i},X_i)\big).$$
and $g_1,\dots,g_m\in\Hom_{\Kb_{\Lambda}}(X_{i},X_{1-i})$ be such that their classes constitute a basis of
$$\Hom_{\Kb_{\Lambda}}(X_{i},X_{1-i})/\big(\Hom_{\Kb_{\Lambda}}(X_{i},X_{1-i})J_{\End_{\Kb_{\Lambda}(X_{i})}}\big).$$
Then we define $\mu^+_{X_i}(X_{1-i})$ and $\mu^-_{X_{i}}(X_{1-i})$ by the triangles
$$
X_{1-i}\xrightarrow{\tiny\begin{pmatrix}f_1\\\vdots\\f_n\end{pmatrix}}X_i^n\rightarrow \mu^+_{X_i}(X_{1-i})\mbox{ and }\mu^-_{X_{i}}(X_{1-i})\rightarrow X_{i}^m\xrightarrow{\tiny\begin{pmatrix}g_1&\cdots&g_m\end{pmatrix}}X_{1-i}.
$$
Thus, one can note that $\mu^+_{X_i}(X_{1-i})$ is a cone of {\it a minimal right approximation} and $\mu^-_{X_i}(X_{1-i})$  a cocone of {\it a minimal left approximation} of $X_{1-i}$ with respect to $X_i$.
Then $X_i\oplus\mu^+_{X_i}(X_{1-i})$ and $\mu^-_{X_{1-i}}(X_i)\oplus X_{1-i}$ are silting complexes. It follows from the results of \cite{AIR} that if $X\not=\Lambda,\Lambda[1]$, then  for exactly one $i\in\{0,1\}$ both of the obtained complexes can be represented by a two-term complex in $\Kbp_\Lambda$.
We choose such $i$ and introduce two-term silting complexes $\mu^+(X)=X_i\oplus\mu^+_{X_i}(X_{1-i})$ and $\mu^-(X)=\mu^-_{X_{1-i}}(X_i)\oplus X_{1-i}$. We will call $\mu^+(X)$ and $\mu^-(X)$ {\it mutations} of the complex $X$. Suppose that $\Lambda=P_1\oplus P_2$, where $P_1$ and $P_2$ are indecomposable projective $\Lambda$-modules. The mutations of $\Lambda$ are by definition the complexes $P_1\oplus\mu_{P_1}^+(P_2)$ and $\mu^+_{P_1}(P_2)\oplus P_2$ and we denote both of them by $\mu^+(\Lambda)$. Analogously, the mutations of $\Lambda[1]$ are $P_1\oplus\mu_{P_1}^-(P_2)$ and $\mu^-_{P_1}(P_2)\oplus P_2$ and we denote both of them by $\mu^-(\Lambda[1])$. For any two-term silting complex $X$, one has $\mu^+\mu^-(X)=\mu^-\mu^+(X)=X$ whenever everything is well defined after an appropriate choice of $\mu^-$ and $\mu^+$ if this is needed.

Any two-term silting complex $X$ is a sum of two indecomposable two-term presilting complexes and each of these complexes has the form $X_i=P_1^{a_i}\oplus P_2^{b_i}\rightarrow P_1^{c_i}\oplus P_2^{d_i}$ ($i=0,1$), where either $a_0=a_1=0$ or $c_0=c_1=0$ and either $b_0=b_1=0$ or $d_0=d_1=0$.
We will call the vector $(c_i-a_i,d_i-b_i)$ the {\it $g$-vector} of $X_i$ and will denote it by $g_{X_i}$.
Then we get the point $O_{X_i}=\frac{g_{X_i}}{|g_{X_i}|}$ on the unit circle.
The points $O_{X_0}$ and $O_{X_1}$ divide the unit circle into two arcs. We associate with $X$ the shorter of these arcs and denote it by $L_X$.
Due to \cite{AIR} there are two-term silting complexes $Y^0$ and $Y^1$ that can be obtained via mutation process from $X$ such that $L_{Y^i}\cap L_X=O_{X_i}$ for $i=0,1$.
Moreover, it is shown in \cite{DIJ} that for a  two-term silting complex $Y\not=X,Y^0,Y^1$ one has $L_X\cap L_Y=\varnothing$.

\section{Quasi-hereditary algebras with two simple modules}

Let us consider the quiver $Q^{p,q}$ ($p,q>0$) with two vertices $1$ and $2$, $p$ arrows $\alpha_1,\dots,\alpha_p$ from the vertex $1$ to the vertex $2$ and $q$ arrows $\beta_1,\dots,\beta_q$ from the vertex $2$ to the vertex $1$. Let $I^{p,q}$ be the ideal of $\kk Q^{p,q}$ generated by the paths $\beta_j\alpha_i$ ($1\le i\le p$, $1\le j\le q$).
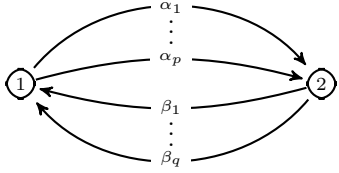
\begin{center}

\begin{tikzpicture}[->,>=stealth',shorten >=0.05cm,auto,node distance=1cm,
                    thick,main node/.style={rectangle,draw,fill=white!10,rounded corners=1.5ex,font=\sffamily \scriptsize \bfseries },rigid node/.style={rectangle,draw,fill=black!20,rounded corners=1.5ex,font=\sffamily \scriptsize \bfseries },style={draw,font=\sffamily \scriptsize \bfseries }]

\node[main node] (1) {$1$};
\node[] (1r) [right   of=1]{};
\node[] (1rr) [right   of=1r]{};
\node[] (1rrr) [right   of=1rr]{};
\node[main node] (2) [right   of=1rrr]{$2$};
 
\path[every node/.style={font=\sffamily\small}]
(1) edge  [bend right=-50] node[above=-6, fill=white]{\tiny  $\alpha_1$  } (2)
(1) edge  [bend right=-15] node[above=-6, fill=white]{\tiny  $\alpha_p$  } node[above=2]{\tiny\bf  $\vdots$  }(2)
(2) edge  [bend right=-15] node[above=-6, fill=white]{\tiny  $\beta_1$  } (1) 
(2) edge  [bend right=-50] node[above=-6, fill=white]{\tiny  $\beta_q$  } node[above=2]{\tiny\bf  $\vdots$  } (1)
;
\end{tikzpicture}
\end{center}
Our main object of interest are the algebras $\Lambda^{p,q}=\kk Q^{p,q}/I^{p,q}$. It is proved in \cite{MH} that all the algebras $\Lambda^{p,q}$ are quasi-hereditary and that any basic quasi-hereditary algebra with two simple modules over an algebraically closed field is isomorphic to $\Lambda^{p,q}$ for some integers $p$ and $q$. The algebra $\Lambda^{0,0}$ is semisimple and does not deserve our attention.  The algebra $\Lambda^{p,0}$ is the $p$-Kronecker algebra. Its derived Picard group was calculated in \cite[Theorem 4.3]{MY}. Since the $p$-Kronecker algebra is hereditary, all indecomposable complexes over it have at most two nonzero terms (see \cite{Hap}) and one can show that all tilting complexes have one or two nonzero terms too.
By these reasons, we will consider only the case $p,q>0$ in the next section. Note that, under this assumption, all the algebras $\Lambda^{p,q}$ are pairwise nonisomorphic.
We will describe the derived Picard group of the algebra $\Lambda^{p,q}$. Our description is based on the fact that any tilting complex can be sent to a tilting complex of length not greater than $2$ via the Serre functor on the category $\Kb_{\Lambda^{p,q}}$.
We also classify two-term silting complexes over $\Lambda^{p,q}$ and describe all the algebras derived equivalent to $\Lambda^{p,q}$.

As a quasi-hereditary algebra, $\Lambda^{p,q}$ has a sequence of standard objects and a sequence of costandard objects. One can easily check that the sequence of standard objects for $\Lambda^{p,q}$ is $(S_2,P_1)$ and the sequence of costandard objects is $(I_1,S_2)$, where we denote by $S_i$, $P_i$, and $I_i$ the simple, the projective and the injective modules corresponding to the vertex $i$ respectively. We will use the same notation for modules over the algebra $\Lambda^{p,q}$ and  for modules over the algebra $\Lambda^{q,p}$. Which category is meant will be each time clear from the context.

\section{Derived Picard group}

In this section we describe the derived Picard group of $\Lambda^{p,q}$. Let us first recall the definition of Ringel duality, specialize it to $\Lambda^{p,q}$ and give a consequence of the results of \cite{DUB} that will play a crucial role in the current paper.

For a quasi-hereditary algebra $\Lambda$, there exists a  quasi-hereditary algebra $\Gamma$ with a tilting $\Gamma$-module $T$ such that $\End_{\Gamma}(T)\cong\Lambda$ and $\Filt(\Delta)\cap \Filt(\nabla)=\add T$, where $\Filt(\Delta)$ is the subcategory of $\Gamma$-modules admitting a filtration by standard modules and $\Filt(\nabla)$ is  the subcategory of $\Gamma$-modules admitting a filtration by costandard modules (see \cite{Rin}). This tilting module gives the Ringel duality functor $\omega:\Db_\Lambda\rightarrow \Db_{\Gamma}$ that sends standard objects of $\Lambda$ to costandard objects of $\Gamma$.
The algebra $\Gamma$ is called {\it the Ringel dual} of $\Lambda$.
If $\Lambda=\Lambda^{p,q}$, then a direct calculation shows that $\Lambda^{\rm op}=\Lambda^{q,p}$ is the Ringel dual of $\Lambda$ (see \cite{LY}) and the $\Lambda^{q,p}$-module $T$ is a direct sum of the module $S_2$ and the cokernel of the map $P_1^{q^2-1}\rightarrow P_2^{q}$ that sends the $n$-th standard generator of $P_1$ with $n=qi+j<q^2$ ($0\le i\le q-1$, $1\le j\le q$) to $\iota_i(\alpha_j)$ if $i\not=j$ and to $\iota_i(\alpha_j)-\iota_q(\alpha_q)$ if $i=j$. Here we denote by $\iota_i:P_2\rightarrow P_2^{q}$ the canonical $i$-th embedding.
We will write $\omega_{p,q}$ instead of simply $\omega$ to emphasize that we mean the Ringel duality functor for $\Lambda^{p,q}$.
 Note that $S_2$ is the cokernel of the map $P_1^q\xrightarrow{(\alpha_1,\dots,\alpha_q)} P_2$. Thus, $\omega_{p,q}$ sends $P_1$ to $P_1^q\rightarrow P_2$ and $P_2$ to $P_1^{q^2-1}\rightarrow P_2^{q}$.  It also can be shown that the inverse of the Ringel duality functor from $\Lambda^{q,p}$ to $\Lambda^{p,q}$ which we will denote by $\omega_{q,p}^{-1}$ sends $P_1$ to $P_2\rightarrow P_1^p$ and $P_2$ to $P_2^{p}\rightarrow P_1^{p^2-1}$. In the first case the corresponding tilting complex is concentrated in degrees $0$ and $-1$ and in the second case it is concentrated in degrees $0$ and $1$, but it is not very important for us at this point. Note also that $\omega_{q,p}^{-1}(S_2)=P_1$.

Let us introduce the derived autoequivalence $\nu_{p,q}$ of $\Lambda^{p,q}$ by the equality $\nu_{p,q}=\omega_{q,p}\omega_{p,q}$. Then it follows, for example, from \cite{Kra} that $\nu_{p,q}$ is a Serre functor. Then the following lemma follows directly from \cite[Theorem 3.3]{DUB}.

\begin{lemma}\label{reduc} Let $X$ be a tilting complex. Then there is some $m\in\mathbb{Z}$ such that $\nu_{p,q}^m(X)$ can be represented by a two-term complex of projective modules.
\end{lemma}
\begin{proof} By \cite[Theorem 3.3]{DUB}, the complex $\nu_{p,q}^m(X)$ is isomorphic to a $\Lambda^{p,q}$-module for some $m\in\mathbb{Z}$. By the same theorem, either projective or injective dimension of the obtained module equals $1$. If the projective dimension is one, then we are done. In the other case, $\nu_{p,q}^{m-1}(X)$ can be represented by a two-term complex of projective modules because $\nu_{p,q}^{-1}$ sends injective modules to projective.
\end{proof}

Let us now classify all two-term tilting complexes giving autoequivalences of the derived category of $\Lambda^{p,q}$. Note that such a complex has an indecomposable direct summand with endomorphism algebra $\kk$.
Let us recall that any object of $\Kbp_{\Lambda}$ can be represented by a unique up to isomorphism in the category $\Cb_{\Lambda}$ complex $(X,d)$ such that $\Im d\subset XJ_\Lambda$. Such a complex $X$ is called {\it a radical complex}.
We will use the general fact that if $X$ is a radical silting complex with nonzero components in the range $[m,n]$ with $m<n$, then $X_m$ and $X_n$ cannot contain a common nonzero direct summand. Indeed, if $\iota:P\hookrightarrow X_n$ is a direct embedding and $\pi:X_m\twoheadrightarrow P$ is a projection on a direct summand, then the morphism from $X$ to $X[n-m]$ that equals $\iota\pi$ in the degree $n$ and equals $0$ in all the other degrees represents a nonzero morphism in $\Kb_{\Lambda}$. It follows from the fact that $\Im(\iota\pi)\not\subset X_mJ_\Lambda$ and any null homotopic morphism has to have image in $XJ_\Lambda$.

\begin{lemma}\label{P1} Suppose that $X$ is a radical indecomposable pretilting complex over $\Lambda=\Lambda^{p,q}$ concentrated in degrees $0$ and $-1$ such that $\End_{\Kb_\Lambda}(X)\cong \kk$. Then $X$ is one of the complexes $P_1$, $P_1[1]$, $\omega_{q,p}(P_1)$ and $\omega_{p,q}^{-1}(P_1)[1]$.
\end{lemma}
\begin{proof} Suppose first that $X$ has the form $P_1^a\rightarrow P_2^b$. We may assume that $b>0$. Since $\Hom_{\Kb_\Lambda}(X,X[1])=\Hom_{\Kb_\Lambda}(X,X[-1])=0$, by \cite{Hap} we have
\begin{equation}\label{eq_1}
1=\dim_\kk\End_{\Kb_\Lambda}(X)=a^2+b^2(1+pq)-ab(p+q).
\end{equation}
Note now that $\dim_\kk\Hom_{\Cb_\Lambda}(X,X[1])=\dim_\kk\Hom_\Lambda(P_1^a,P_2^b)=abp$. On the other hand, any radical map from $P_2^b$ to $P_2^b$ is annihilated by the differential of $X$, and hence the number of linearly independent null homotopic maps from $X$ to $X[1]$ is not greater than $a^2+b^2-1$. We subtract $1$ because there is the identity map from $X$ to $X$ that gives the zero null homotopic map. Thus, we have $abp\le a^2+b^2-1$. Subtracting \eqref{eq_1} from the obtained inequality, we get $bp\le a$.

Note that the map from $P_1^a$ to $P_2^b$ has to be injective because in the opposite case its kernel would be a direct summand of $X$. We have $a(q+1)=\dim_\kk P_1^a\le\dim_\kk P_2^bJ_{\Lambda}=bp(q+1)$, and hence $a\le bp$. Thus, we have $a=bp$ and it follows from \eqref{eq_1} that $X=\omega_{q,p}(P_1)$.

If $X$ has the form $P_2^a\rightarrow P_1^b$, then the complex $\Hom_{\Lambda^{p,q}}(X[-1],\Lambda^{p,q})$ of $\Lambda^{q,p}$-modules  has the form $P_1^b\rightarrow P_2^a$, and hence it is isomorphic to $P_1[1]$ or $\omega_{p,q}(P_1)$ by the argument above. Then $X$ is isomorphic to $P_1$ or $\omega_{p,q}^{-1}(P_1)[1]$.
\end{proof}

Let us recall that, for a finite dimensional algebra $\Lambda$, the Picard group $\Pic(\Lambda)$ is the group of autoequivalences of the category of $\Lambda$-modules modulo natural isomorphisms. If $\Lambda$ is basic, then this group is isomorphic to the group of outer automorphisms $\Out(\Lambda)=\Aut(\Lambda)/\Inn(\Lambda)$. Here $\Aut(\Lambda)$ is the group of automorphisms of $\Lambda$ and $\Inn(\Lambda)$ is the group of inner automorphisms of $\Lambda$. This isomorphism is induced by the map $\Aut(\Lambda)\rightarrow \Pic(\Lambda)$ that sends an automorphism $\theta:\Lambda\rightarrow\Lambda$ to the autoequivalence $-\otimes_{\Lambda}\Lambda_{\theta^{-1}}$. Here $\Lambda_{\theta^{-1}}$ is  the bimodule coinciding with $\Lambda$ as a left module and having the right multiplication $*$ by elements of $\Lambda$ defined by the equality $x*a=x\theta^{-1}(a)$, where the multiplication on the right side is the original multiplication of $\Lambda$. The group $\Pic(\Lambda)$ is a subgroup of $\TrPic(\Lambda)$ in a natural way. Moreover, an element of $\TrPic(\Lambda)$ belongs to $\Pic(\Lambda)$ if and only if the radical representative of the corresponding tilting complex is concentrated in degree zero. In fact, the autoequivalence $-\otimes_{\Lambda}\Lambda_{\theta^{-1}}$ can be defined by the tilting complex $\Lambda$ and the isomorphism from $\Lambda$ to $\End_{\Lambda}(\Lambda)$ that sends $x\in\Lambda$ to the left multiplication by $\theta(x)$.

\begin{corollary}\label{grprod} Suppose that $\Lambda=\Lambda^{p,q}$ for some $p,q>0$. Let $H\cong \Pic(\Lambda)\times\mathbb{Z}$ be the subgroup of $\TrPic(\Lambda)$ generated by the Picard group of $\Lambda$ and the shift and $K\cong \mathbb{Z}$ be the subgroup generated by $\omega_{p,p}$ in the case $p=q$ and by $\nu_{p,q}$ in the case $p\not=q$.
Then $K\cap H=\{\Id\}$ and $KH=\TrPic(\Lambda)$.
\end{corollary}
\begin{proof} Let us prove that $K\cap H=\{\Id\}$. Since for any $F\in H$ the radical representative of $F(\Lambda)$ is concentrated in one degree, it is enough to prove for the generator $\rho$ of the group $K$ that $\rho^t(\Lambda)$ has length more than one for any $t>0$. Note that the radical representative of $\omega_{p,q}(\Lambda)$ is concentrated in degrees $-1$ and $0$ and has the form $P_1^a\rightarrow P_2^b$. Then it is enough to prove that, for any radical complex $X$ concentrated in the interval $[-l,0]$ such that $X_0=P_2^b$ and $X_{-l}=P_1^a$ for some $a,b>0$, the radical representative $Y$ of $\omega_{p,q}(X)$ is concentrated in the interval $[-l-1,0]$ and has $Y_0=P_2^c$ and $Y_{-l-1}=P_1^d$ for some $d,c>0$. Everything is clear from the definition of $\omega_{p,q}$ except the fact $c,d>0$. But the last assertion follows from the fact that $\Hom_{\Kb_\Lambda}\big(\omega_{p,q}^{-1}(\Lambda),X\big)$ and $\Hom_{\Kb_\Lambda}\big(X,\omega_{p,q}^{-1}(\Lambda[l+1])\big)$ are nonzero.

It remains to prove that $KH=\TrPic(\Lambda)$.
Let us pick some $F\in\TrPic(\Lambda)$ and denote by $X$ the radical representative of $F(\Lambda)$.
It follows from Lemma \ref{reduc} that, for some $m\in\mathbb{Z}$, the radical representative of $\nu_{p,q}^m(X)$ has length not greater than two. Thus, we may assume that $X$ has length  not greater than two. Since $X$ has the direct summand $F(P_1)$ with endomorphism algebra $\kk$ in the homotopy category, we can apply some shift and one of the equivalences $\omega_{p,q}$ or $\omega_{q,p}^{-1}$ to $X$ and get a tilting complex with a direct summand isomorphic to $P_1$ by Lemma \ref{P1}. Thus, we may assume that either $F(P_1)=P_1$ or $\omega_{p,q}F(P_1)=P_1$.
Let $Y$ be the radical representative of $F(P_2)$ or $\omega_{p,q}F(P_2)$ respectively. Suppose that $Y$ is concentrated in the interval $[r,s]$ such that $Y_r,Y_s\not=0$.
If $Y_r$ has the form $P_1^a$ with $a>0$, then $\Hom_{\Kb_{\Lambda'}}(Y,P_1[-r])\not=0$. If $Y_r$ has the form $P_2^a$ with $a>0$, then the map $\alpha_1:P_1\rightarrow P_2$ gives a nonzero element of $\Hom_{\Kb_{\Lambda'}}(P_1[-r],Y)$ because the kernel of any radical map from $P_2$ to a projective module contains the image of $\alpha_1$. Here $\Lambda'$ denotes either $\Lambda$ or $\Lambda^{\rm op}$ depending on what category the complex $Y$ belongs to. Thus, we have $r=0$. Analogously, $s=0$. Thus, $X$ or $\omega_{p,q}(X)$ is concentrated in degree zero. In the first case we have $F\in \Pic(\Lambda)$. In the second case we have an equivalence $\omega_{p,q}F:\Kb_\Lambda\rightarrow \Kb_{\Lambda^{\rm op}}$ that sends $\Lambda$ to $\Lambda^{\rm op}$. Thus, this case is possible only when $\Lambda\cong\Lambda^{\rm op}$, i.e. $p=q$. If $p=q$ and $\omega_{p,p}(X)$ is concentrated in degree $0$, then $\omega_{p,p}F\in \Pic(\Lambda)$ and the corollary is proved.
\end{proof}

It remains to calculate $\Pic(\Lambda^{p,q})\cong \Out(\Lambda^{p,q})$ and derive the commutation formulas for $\omega_{p,q}$, $\omega_{q,p}$ and elements of $\Pic(\Lambda^{p,q})$ and $\Pic(\Lambda^{q,p})$. The next lemma implements the first part of this plan.
Let us introduce the vector spaces ${\bA}={\bA}^{p}=e_2\Lambda^{p,q}e_1=\oplus_{i=1}^p\kk\alpha_i$ and ${\bB}={\bB}^{q}=e_1\Lambda^{p,q}e_2=\oplus_{i=1}^q\kk\beta_i$. We identify ${\bA}$ and ${\bB}$ with the corresponding subspaces of $\Lambda^{p,q}$.
Moreover, we identify the elements of $\GL({\bA})\times \GL({\bB})$ with the automorphisms of $\Lambda^{p,q}$ induced by them. We also denote by $D=D^{p,q}$ the subgroup of $\GL({\bA})\times \GL({\bB})$ formed by the elements of the form $(\lambda\Id_{{\bA}},\lambda^{-1}\Id_{{\bB}})$ for $\lambda\in\kk^*$.

\begin{lemma}\label{pic} $\Out(\Lambda^{p,q})\cong \big(\GL({\bA})\times \GL({\bB})\big)/D$. This isomorphism is induced by the canonical projection $\Aut(\Lambda^{p,q})\twoheadrightarrow \Out(\Lambda^{p,q}))$.
\end{lemma}
\begin{proof} It follows from \cite{Pol} and \cite{GAS} that $\Out(\Lambda^{p,q})\cong \Aut_S(\Lambda^{p,q})/\big(\Inn_S(\Lambda^{p,q})\cap \Aut_S(\Lambda^{p,q})\big)$, where $\Aut_S(\Lambda^{p,q})$ denotes the set of automorphisms of $\Lambda^{p,q}$ that stabilize the subalgebra generated by $e_1$ and $e_2$. It is clear that all such isomorphisms act identically on $e_1$ and $e_2$, and hence have the form $(g,h)$ for some $g\in\GL({\bA})$ and $h\in\GL({\bB})$. It remains to show that $\Inn_S(\Lambda^{p,q})\cap \Aut_S(\Lambda^{p,q})=D$. But it follows from the fact that any invertible $x$ such that $x^{-1}e_1x=e_1$ and
$x^{-1}e_2x=e_2$ has the form $x=\lambda_1e_1+\lambda_2e_2+\sum\limits_{1\le i\le p,1\le j\le q}\kappa_{i,j}\alpha_i\beta_j$ for some $\lambda_1,\lambda_2\in\kk^*$ and $\kappa_{i,j}\in\kk$. Then modulo the center of $\Lambda^{p,q}$ the element $x$ has the form $x=\lambda e_1+e_2$ for some $\lambda\in\kk^*$. It is easy to see that the inner automorphisms induced by such elements $x$ are exactly automorphisms from $D$.
\end{proof}

For the commutator formula, we will need the description of $\omega_{p,q}$ on morphisms. We will use for this a description of $\omega_{p,q}$ that is a little different from the one we used before. Let us introduce the $\Lambda^{p,q}$-module $\Xi=\Xi^{p,q}$. As usually, to do this we describe the spaces $\Xi_1=\Xi e_1$ and $\Xi_2=\Xi e_2$ and the maps $\phi_{1,2}:{\bB}\rightarrow \Hom_\kk(\Xi_1,\Xi_2)$ and $\phi_{2,1}:{\bA}\rightarrow \Hom_\kk(\Xi_2,\Xi_1)$ induced by the multiplication by the arrows of $Q^{p,q}$. We define $\Xi_1=\kk$ and $\Xi_2={\bA}^*\oplus {\bB}=\Hom_\kk({\bA},\kk)\oplus {\bB}$. We define $\phi_{1,2}(v)(1)=(0,v)$ and $\phi_{2,1}(u)(f,v)=f(u)$ for $u\in{\bA}$, $v\in{\bB}$ and $f\in{\bA}^*$. Now we have natural isomorphisms $\Hom_{\Lambda^{p,q}}(S_2,\Xi)\cong {\bB}$ and $\Hom_{\Lambda^{p,q}}(\Xi,S_2)\cong ({\bA}^*)^*\cong {\bA}$. We will identify the corresponding spaces via these isomorphisms. In particular, $\GL({\bA})$ acts on $\Hom_{\Lambda^{p,q}}(\Xi,S_2)$ and $\GL({\bB})$ acts on $\Hom_{\Lambda^{p,q}}(S_2,\Xi)$. Now it is easy to see that actually $\omega_{p,q}(P_1)$ is a minimal $\Lambda^{q,p}$-projective resolution of $S_2$ and $\omega_{p,q}(P_2)$ is a minimal $\Lambda^{q,p}$-projective resolution of $\Xi^{q,p}$. Thus, replacing $\omega_{p,q}$ by a naturally isomorphic equivalence, we may assume that $\omega_{p,q}(\Lambda^{p,q})=S_2\oplus \Xi^{q,p}$ and the isomorphism $\Lambda^{p,q}\cong \End_{\Lambda^{q,p}}(S_2\oplus \Xi^{q,p})$ determined by $\omega_{p,q}$ is induced by the isomorphisms ${\bA}^{p}\cong {\bB}^{p}=\Hom_{\Lambda^{q,p}}(S_2,\Xi^{q,p})$ and ${\bB}^{q}\cong {\bA}^{q}=\Hom_{\Lambda^{q,p}}(\Xi^{q,p},S_2)$. Here the isomorphisms ${\bA}^{p}\cong {\bB}^{q}$ and ${\bB}^{p}\cong {\bA}^{q}$ are the {\it renaming isomorphisms}
sending $\alpha_i$ to $\beta_i$ and $\beta_j$ to $\alpha_j$ for $1\le i\le p$ and $1\le j\le q$. Note that the renaming isomorphisms induce the isomorphism 
$$
\Phi_{p,q}:\GL({\bA}^{p})\times \GL({\bB}^{q})\cong \GL({\bA}^{q})\times \GL({\bB}^{p})
$$
that induces an isomorphism from $\Pic(\Lambda^{p,q})$ to $\Pic(\Lambda^{q,p})$. We denote the induced isomorphism by $\Phi_{p,q}$ too. Note that the composition $\Phi_{q,p}\Phi_{p,q}$ is the identity automorphism.
In the case $p=q$ we get an automorphism $\Phi_{p,p}\in\Aut\big(\Pic(\Lambda^{p,p})\big)$ of order two interchanging two copies of $\GL_p(\kk)\cong \GL({\bA}^{p})\cong \GL({\bB}^{p})$ and the corresponding semidirect product $\Pic(\Lambda^{p,p})\rtimes_{\Phi_{p,p}}\mathbb{Z}$ with the multiplication defined by the equality $(F,a)*(G,b)=(F\Phi_{p,p}^a(G),a+b)$ for $a,b\in\mathbb{Z}$ and $F,G\in \Pic(\Lambda^{p,p})$.

\begin{lemma}\label{comm} $\omega_{p,q}F\cong\Phi_{p,q}(F)\omega_{p,q}$ for any $F\in \Pic(\Lambda^{p,q})$.
\end{lemma}
\begin{proof} Let $\Lambda$ denote $\Lambda^{p,q}$ and $\Phi$ denote $\Phi_{p,q}$.   We need to construct an isomorphism $\phi:\omega_{p,q}F(\Lambda)\cong \Phi(F)\omega_{p,q}(\Lambda)$ such that $\phi\circ \omega_{p,q}F(x)=\Phi(F)\omega_{p,q}(x)\circ\phi$ for any $x\in\End_{\Lambda}(\Lambda)=\Lambda$.
Let us introduce  $\theta\in\Aut(\Lambda)$  such that $F\cong -\otimes_\Lambda\Lambda_{\theta^{-1}}$.
We have by our definitions $\omega_{p,q}F(\Lambda)=S_2\oplus \Xi^{q,p}$ and $\omega_{p,q}F(x)=\omega_{p,q}\theta(x)=\Phi(\theta)\omega_{p,q}(x)$ for $x\in {\bA}^{p}\oplus {\bB}^{q}$.
Here $\omega_{p,q}$ is the renaming isomorphism from  ${\bA}^{p}\oplus {\bB}^{q}$ to ${\bA}^{q}\oplus  {\bB}^{p}=\Hom_{\Lambda^{q,p}}(\Xi^{q,p},S_2)\oplus \Hom_{\Lambda^{q,p}}(S_2,\Xi^{q,p})$ and $\Phi$ is the conjugation by this isomorphism.

Applying again our definitions we get $\Phi(F)\omega_{p,q}(\Lambda)=(S_2\oplus \Xi^{q,p})_{\Phi(\theta^{-1})}$ and $\Phi(F)\omega_{p,q}(x)=\omega_{p,q}(x)$ for $x\in {\bA}^{p}\oplus {\bB}^{q}$. Here $\omega_{p,q}(x):(S_2\oplus \Xi^{q,p})_{\Phi(\theta^{-1})}\rightarrow (S_2\oplus \Xi^{q,p})_{\Phi(\theta^{-1})}$ is the map obtained by the identification of the linear spaces $(S_2\oplus \Xi^{q,p})_{\Phi(\theta^{-1})}$ and $S_2\oplus \Xi^{q,p}$.

Let us now introduce $\phi:S_2\oplus \Xi^{q,p}\rightarrow S_2\oplus \Xi^{q,p}$ that is identical on $S_2$ and is defined on $\Xi^{q,p}=({\bA}^{q})^*\oplus {\bB}^{p}\oplus \kk$ by the equality $\phi(f,v,\lambda)=\big(f\circ \Phi(\theta), \Phi(\theta^{-1})(v),\lambda\big)$ for $(f,v,\lambda)\in ({\bA}^{q})^*\oplus {\bB}^{p}\oplus \kk$. Now, for $x=(u,w)\in {\bA}^{q}\oplus {\bB}^{p}\subset \Lambda^{q,p}$, we get
\begin{multline*}
\phi\big((f,v,\lambda)x\big)=\phi\big(0,\lambda w, f(u)\big)=\big(0, \lambda \Phi(\theta^{-1})(w),f(u)\big)\\
=(f\circ \Phi(\theta),\Phi(\theta^{-1})(v),\lambda)\Phi(\theta^{-1})(x)=\phi(f,v,\lambda)\Phi(\theta^{-1})(x),
\end{multline*}
i.e. $\phi$ induces an isomorphism $S_2\oplus \Xi^{q,p}\cong (S_2\oplus \Xi^{q,p})_{\Phi(\theta^{-1})}$. It remains to check that $\phi\circ \Phi(\theta)\omega_{p,q}(x)=\omega_{p,q}(x)\circ\phi$ for $x\in {\bA}^{p}\cup {\bB}^{q}$.
Let $e$ be the generator of $S_2$.
For $u\in {\bA}^{p}$, $w\in {\bB}^{q}$, and $(f,v,\lambda)\in ({\bA}^{q})^*\oplus {\bB}^{p}\oplus \kk$, we have
\begin{multline*}
\big(\phi\circ \Phi(\theta)\omega_{p,q}(u)\big)(e)=\phi\big(0,\Phi(\theta)\omega_{p,q}(u),0\big)\\
=\big(0,\omega_{p,q}(u),0\big)=\omega_{p,q}(u)(e)=\big(\omega_{p,q}(u)\circ\phi\big)(e),
\end{multline*}
\begin{multline*}
\big(\phi\circ \Phi(\theta)\omega_{p,q}(w)\big)(f,v,\lambda)=\big(f\Phi(\theta)\omega_{p,q}(w)\big)\phi(e)=\big(f\Phi(\theta)\omega_{p,q}(w)\big)e\\
=\omega_{p,q}(w)\big(f\Phi(\theta), \Phi(\theta^{-1})(v),\lambda\big)=\big(\omega_{p,q}(w)\circ\phi\big)(f,v,\lambda),
\end{multline*}
and thus the lemma is proved.
\end{proof}

Now we are ready to formulate one of our main results. Let us introduce the group $G_{p,q}(\kk)=\big(\GL_p(\kk)\times \GL_q(\kk)\big)/D_{p,q}(\kk)$, where $D_{p,q}(\kk)$ is the subgroup formed by the elements of the form $(\lambda\Id_{\kk^p},\lambda^{-1}\Id_{\kk^q})$ for $\lambda\in\kk^*$.
In the case $p=q$ we denote by $\Phi$ the automorphism of $G_{p,q}(\kk)$ induced by the interchanging of two copies of $\GL_p(\kk)$. This automorphism as before gives an action of $\mathbb{Z}$ on $G_{p,q}(\kk)$, and hence determines the semidirect product $G_{p,q}(\kk)\rtimes_{\Phi}\mathbb{Z}$.

\begin{theorem} Let $p\not=q$ be positive integers. Then $\TrPic(\Lambda^{p,q})\cong G_{p,q}(\kk)\times\mathbb{Z}\times\mathbb{Z}$ and $\TrPic(\Lambda^{p,p})\cong (G_{p,p}(\kk)\rtimes_{\Phi}\mathbb{Z})\times\mathbb{Z}$.
\end{theorem}
\begin{proof} The theorem follows directly from Corollary \ref{grprod} and Lemmas \ref{pic} and \ref{comm}.
\end{proof}

\section{Two-term silting complexes and derived equivalences}\label{2term}

In this section we describe all two-term silting complexes over $\Lambda^{p,q}$. By Lemma \ref{reduc}, this will allow us to describe all the algebras derived equivalent to $\Lambda^{p,q}$.
Note that the last point can be achieved using \cite[Theorem 4.7]{DUB}, but in any case to get the description in terms of quivers with relations, one has to obtain projective resolutions of the modules $\nu_{p,0}^m(\Lambda^{p,0})$ that is of the same difficulty as the direct description of  two-term silting complexes over $\Lambda^{p,q}$.

Let us fix some nonnegative integers $p,q$. Assume that $p\ge 2$. Let us introduce the sequence of $\kk$-linear spaces ${\bA}_i={\bA}^p_i$  with monomorphisms $\kappa_k=\kappa_{i,k}:{\bA}_i\hookrightarrow {\bA}_{i+1}$ ($i\ge 0$, $1\le k\le p$)  by induction.
We set $\bA_0=\kk$, $\bA_1=\oplus_{i=1}^p\kk\alpha_i\cong\kk^p$ and define $\kappa_{0,k}:\bA_0\rightarrow \bA_1$ by the equality $\kappa_{0,k}(1)=\alpha_k$ for $1\le k\le p$. Suppose that we have already defined $\bA_{m-1}$, $\bA_m$ and $\kappa_{m-1,k}:\bA_{m-1}\hookrightarrow \bA_m$ for some integer $m$ and all $1\le k\le p$. Let $\iota_m$ denote the monomorphism $$\bA_{m-1}\cong \bA_{0}\otimes\bA_{m-1}\xrightarrow{\sum\limits_{k=1}^p\kappa_{0,k}\otimes\kappa_{m-1,k}} \bA_{1}\otimes \bA_m.$$ Here and further we write simply $\otimes$ instead of $\otimes_\kk$. We set $\bA_{m+1}=\Coker\iota_m$ and denote by $\pi_m:\bA_{1}\otimes \bA_m\twoheadrightarrow \bA_{m+1}$ the canonical projection. Now we define $\kappa_{m,k}$ for $1\le k\le p$ as the composition $$\bA_{m}\cong \bA_{0}\otimes\bA_m\xrightarrow{\kappa_{0,k}\otimes\Id_{\bA_{m}}} \bA_{1}\otimes \bA_m\xrightarrow{\pi_m}\bA_{m+1}.$$ For convenience, we also set $\bA_{-1}=0$. Let us now introduce $a_m=a_{p,m}:=\dim_\kk\bA_m$. Then the numbers $a_k$ satisfy the recursive formula $a_{m+1}=pa_m-a_{m-1}$ and induction argument shows that $a_m^2+a_{m-1}^2-pa_{m-1}a_m=1$. Moreover, it is not difficult to show that $\frac{a_m}{a_{m-1}}$ is a decreasing sequence with the limit $\frac{p+\sqrt{p^2-4}}{2}$.

Let $C_m=C_m^{p,q}$ ($m\ge 0$) be the two-term $\Lambda^{p,q}$-complex $$\bA_{m-1}\otimes P_1\xrightarrow{\sum\limits_{k=1}^p\kappa_k\otimes\alpha_k}  \bA_m\otimes P_2$$ concentrated in degrees $-1$ and $0$. Note that $C_0=P_2$, $C_1=\left(P_1\xrightarrow{\tiny\begin{pmatrix}\alpha_1\\\vdots\\\alpha_p\end{pmatrix}} P_2^p\right)=\mu^+_{P_2}(P_1)$, and hence $C_0\oplus C_1$ is a silting $\Lambda^{p,q}$-complex. Moreover, it is not difficult to see that $C_0\oplus C_1$ is a tilting complex. We set also $C_{-1}:=P_1$ for convenience.

Let us now introduce the algebras $\Lambda^{p,q}_m$ ($m\ge 0$) in the following way. The algebra $\Lambda^{p,q}_m$ has two primitive orthogonal idempotents $e_1$ and $e_2$ such that $e_1+e_2=1$ and its bimodule structure over the semisimple algebra generated by $e_1$ and $e_2$ is defined by the equalities
\begin{multline*}
e_1\Lambda^{p,q}_me_1=\kk e_1\bigoplus \bA_{m+1}\otimes\bB\otimes \bA_m^*,\ e_1\Lambda^{p,q}_me_2=\bA_{m}\otimes\bB\otimes \bA_m^*\\
e_2\Lambda^{p,q}_me_1=\bA_1\bigoplus \bA_{m+1}\otimes\bB\otimes \bA_{m-1}^*,\ e_2\Lambda^{p,q}_me_2=\kk e_2\bigoplus \bA_{m}\otimes\bB\otimes \bA_{m-1}^*,
\end{multline*}
where, as before, ${\bB}=\oplus_{i=1}^q\kk\beta_i$. The products that do not follow from the $\kk e_1\oplus\kk e_2$-bimodule structure are all zero except the products induced by the maps
\begin{multline*}
\bA_1\otimes(\bA_{m}\otimes\bB\otimes \bA_i^*)\xrightarrow{\pi_m\otimes\Id_{\bB\otimes \bA_i^*}}\bA_{m+1}\otimes\bB\otimes \bA_i^*\,\,\,(i=m,m-1)\mbox{ and}\\
(\bA_{i}\otimes\bB\otimes \bA_m^*)\otimes\bA_1\xrightarrow{\Id_{\bA_i\otimes\bB}\otimes\iota_m^*}\bA_{i}\otimes\bB\otimes \bA_{m-1}^*\,\,\,(i=m,m+1).
\end{multline*}
Here we identify $\bA_m^*\otimes\bA_1$ with $(\bA_1\otimes\bA_m)^*$ via the canonical isomorphism $\bA_m^*\otimes\bA_1^*\cong(\bA_1\otimes\bA_m)^*$ and the isomorphism $\bA_1\cong \bA_1^*$ that sends $\alpha_i$ to $\alpha_i^*$, where $\alpha_1^*,\dots,\alpha_p^*$ is the basis dual to $\alpha_1,\dots,\alpha_p$. Let us note that $\Lambda^{p,q}_m$ is the path algebra of a quiver with vertices $1$ and $2$, $p$ arrows from $1$ to $2$ associated with some basis of $\bA_1$ and $qa_m^2$ arrows from $2$ to $1$ associated with some basis of $\bA_{m}\otimes\bB\otimes \bA_m^*$ modulo the ideal generated by the images of the maps
\begin{multline*}
\bA_{m-1}\otimes\bB\otimes \bA_m^*\xrightarrow{\iota_m\otimes\Id_{\bB\otimes \bA_m^*}}\bA_1\otimes(\bA_{m}\otimes\bB\otimes \bA_m^*)\mbox{ and}\\
\bA_{m}\otimes\bB\otimes \bA_{m+1}^*\xrightarrow{\Id_{\bA_m\otimes\bB}\otimes \pi_m^*}(\bA_{m}\otimes\bB\otimes \bA_m^*)\otimes\bA_1.
\end{multline*}
In particular, the algebra $\Lambda^{p,q}_m$ is quadratic for any $m\ge 0$. Indeed, since $\iota_m^*$ and $\pi_m$ are surjective, it is clear that $J_{\Lambda^{p,q}_m}/J_{\Lambda^{p,q}_m}^2\cong (\bA_{m}\otimes\bB\otimes \bA_{m}^*)\oplus \bA_1$ and it remains to show that
$
\bA_m^*\otimes \Im\iota_m+\Im\pi_m^*\otimes \bA_m=\bA_m^*\otimes\bA_1\otimes\bA_m.
$
Since
$$\dim_\kk(\bA_m^*\otimes \bA_{m-1})+\dim_\kk(\bA_{m+1}^*\otimes \bA_{m})=\dim_\kk(\bA_{m}^*\otimes \bA_1\otimes \bA_{m}),$$ the required equality is equivalent to the fact that $(\iota_m^*\otimes\Id_{\bA_m})(\Id_{\bA_m^*}\otimes\iota_m)$ is an isomorphism and to the fact that $(\Id_{\bA_m^*}\otimes\pi_m)(\pi_m^*\otimes\Id_{\bA_m})$ is an isomorphism.
Thus, we can proceed by induction on $m$ using the equality
$$(\iota_m^*\otimes\Id_{\bA_m})(\Id_{\bA_m^*}\otimes\iota_m)=(\Id_{\bA_{m-1}^*}\otimes\pi_{m-1})(\pi_{m-1}^*\otimes\Id_{\bA_{m-1}}).$$
Since $(\iota_1^*\otimes\Id_{\bA_1})(\Id_{\bA_1^*}\otimes\iota_1)$ is an isomorphism, we are done. Note that $\Lambda^{p,q}_0=\Lambda^{p,q}=\End_{\Kb_{\Lambda^{p,q}}}(C_{-1}\oplus C_0)$.

\begin{lemma} For any $m\ge 0$, the complex $C_{m-1}\oplus C_m$ is a tilting $\Lambda^{p,q}$-complex such that $\End_{\Kb_{\Lambda^{p,q}}}(C_{m-1}\oplus C_m)\cong\Lambda^{p,q}_m$.
\end{lemma}
\begin{proof} Let us set $\Lambda=\Lambda^{p,q}$.
We will proceed by induction on $m$. Suppose that $C_{m-1}\oplus C_m$ is a tilting complex.
Note that any map from $\bA_{m-2}\otimes P_1$ to $\bA_{m-1}\otimes P_1$ is of the form $f\otimes\Id_{P_1}$ for some $f\in\Hom_\kk(\bA_{m-2},\bA_{m-1})$ while any map from $\bA_{m-1}\otimes P_2$ to $\bA_{m}\otimes P_2$ is a sum of a map of the form $f\otimes\Id_{P_2}$ and of maps of the form $f\otimes \alpha_i\beta_j$, where $f\in\Hom_\kk(\bA_{m-1},\bA_{m})$, $1\le i\le p$ and $1\le j\le q$. The maps $f\otimes \alpha_i\beta_j$ automatically give maps from $\Hom_{\Cb_{\Lambda}}(C_{m-1},C_m[-1])$ while the maps of the form $(f_1\otimes\Id_{P_1},f_2\otimes\Id_{P_2})$ have to give $pa_{m-2}a_m$ null homotopic elements of $\Hom_{\Cb_{\Lambda}}(C_{m-1},C_m[1])$. Thus, there are $a_{m-2}a_{m-1}+a_{m-1}a_m-pa_{m-2}a_m=p$ linearly independent maps of the form $(f_1\otimes\Id_{P_1},f_2\otimes\Id_{P_2})$ from $C_{m-1}$ to $C_m$ in $\Kb_{\Lambda}$.
Then it is easy to see that these maps are linear combinations of the maps $(\kappa_k\otimes\Id_{P_1},\kappa_k\otimes\Id_{P_2})$ ($1\le k\le p$). Analogously, one can show that there are no nonzero elements of the form $(f_1\otimes\Id_{P_1},f_2\otimes\Id_{P_2})$ in $\Hom_{\Cb_{\Lambda}}(C_{m},C_{m-1})$.

Now the map $(0,\sum\limits_{i=1}^p\sum\limits_{j=1}^qf_{i,j}\otimes \alpha_i\beta_j)\in\Hom_{\Cb_{\Lambda}}(C_{m-1},C_m[-1])$ is nonzero in $\Kb_{\Lambda}$ if and only if its image does not lie in the image of
$$\sum\limits_{k=1}^p\kappa_k\otimes\alpha_k:\bA_{m-1}\otimes P_1\rightarrow  \bA_m\otimes P_2.$$
Since the image of $\sum\limits_{i=1}^p\sum\limits_{j=1}^qf_{i,j}\otimes \alpha_i\beta_j$ is a submodule of the socle of $\bA_m\otimes P_2$, nonzero elements of $\Hom_{\Kb_{\Lambda}}(C_{m-1},C_m)$ that can be represented by a map of the form $(0,\sum\limits_{i=1}^p\sum\limits_{j=1}^qf_{i,j}\otimes \alpha_i\beta_j)$ correspond to  maps from $\bA_{m-1}\otimes P_2$ to
\begin{multline*}
\Coker\left(\bA_{m-1}\otimes\bA_0\otimes\bB\otimes S_2\xrightarrow{\sum\limits_{k=1}^p\kappa_{m-1,k}\otimes\kappa_{0,k}\otimes\Id_{\bB\otimes S_2}}\bA_m\otimes \bA_1\otimes \bB\otimes S_2\right)\\
\cong \bA_{m+1}\otimes\bB\otimes S_2,
\end{multline*}
i.e. can be naturally parametrized by the elements of $\bA_{m+1}\otimes\bB\otimes \bA_{m-1}^*$.
Analogously, $\Hom_{\Kb_{\Lambda}}(C_{m},C_{m-1})\cong \bA_{m}\otimes\bB\otimes \bA_{m}^*$.

In the same manner one can show that, for $t=m-1,m$, the elements of $\End_{\Kb_{\Lambda}}(C_{t})$ are linear combinations of $\Id_{C_t}$ and maps that can be represented by a map of the form $(0,\sum\limits_{i=1}^p\sum\limits_{j=1}^qf_{i,j}\otimes \alpha_i\beta_j)$ with $f_{i,j}\in\End_\kk(\bA_t)$ and that nonzero maps of the second type are in one to one correspondence with $\bA_{t+1}\otimes\bB\otimes \bA_{t}^*$. Thus, we have isomorphisms $\End_{\Kb_{\Lambda}}(C_{t})\cong \kk\oplus (\bA_{t+1}\otimes\bB\otimes \bA_{t}^*)$. Sending  $(\kappa_k\otimes\Id_{P_1},\kappa_k\otimes\Id_{P_2})$ ($1\le k\le p$) to $\alpha_k$, we get also the isomorphism $\Hom_{\Kb_{\Lambda}}(C_{m-1},C_{m})\cong \bA_1\oplus (\bA_{m+1}\otimes\bB\otimes \bA_{m-1}^*)$. It is clear that all the products not involving $\Id_{C_{m-1}}$ and $\Id_{C_m}$ are zero except the products
\begin{multline*}
\mu_1:\bA_{m}\otimes\bB\otimes \bA_{m}^*\times \bA_1\rightarrow \bA_{m}\otimes\bB\otimes \bA_{m-1}^*,\\
\mu_2:\bA_1\times \bA_{m}\otimes\bB\otimes \bA_{m}^*\rightarrow \bA_{m+1}\otimes\bB\otimes \bA_{m}^*.
\end{multline*}
By our definitions, we have $\mu_1(u\otimes v\otimes f,\alpha_k)=u\otimes v\otimes f\kappa_k$ and $\mu_2(u\otimes v\otimes f,\alpha_k)=\kappa_k(u)\otimes v\otimes f$ for $u\in\bA_m$, $v\in\bB$, $f\in\bA_{m}^*$ and $1\le k\le p$. It remains to note that the map from $\bA_{m}^*\otimes \bA_1$ to $\bA_{m-1}^*$  that sends $f\otimes \alpha_k$ to $f\kappa_k$ is exactly $\iota_m^*$ while the map from $\bA_1\otimes\bA_{m} $ to $\bA_{m+1}$  that sends $\alpha_k\otimes u$ to $\kappa_k(u)$ is exactly $\pi_m$.

Thus, we have proved that $\End_{\Kb_{\Lambda}}(C_{m-1}\oplus C_m)\cong\Lambda^{p,q}_m$, where the arrows from $C_{m-1}$ to $C_m$ correspond to the maps $(\kappa_k\otimes\Id_{P_1},\kappa_k\otimes\Id_{P_2})$ ($1\le k\le p$). Then $\mu^+_{C_m}(C_{m-1})$ is the cone of the map
$$
C_{m-1}\xrightarrow{\tiny\begin{pmatrix}(\kappa_1\otimes\Id_{P_1},\kappa_1\otimes\Id_{P_2})\\\vdots\\(\kappa_p\otimes\Id_{P_1},\kappa_p\otimes\Id_{P_2})\end{pmatrix}} (C_{m})^p
$$
that in turn is isomorphic to $C_{m+1}$. Hence, $C_m\oplus C_{m+1}$ is a two-term silting complex. Since the nonzero component of its differential is injective, it is easy to see that $C_m\oplus C_{m+1}$ is tilting and the induction step is finished.
\end{proof}

If $q\ge 2$, then applying the functor $\Hom_{\Lambda^{q,p}}(-,\Lambda^{q,p}):\Cb_{\Lambda^{q,p}}\rightarrow \Cb_{\Lambda^{p,q}}$ to the tilting complex $C_m^{q,p}[-1]$ ($m\ge -1$), we obtain the tilting $\Lambda^{p,q}$-complex
$$
\left(\bA_{m}^{q,p}\right)^*\otimes P_2\xrightarrow{\sum\limits_{k=1}^q\kappa_k^*\otimes\beta_k}  \left(\bA_{m-1}^{q,p}\right)^*\otimes P_1
$$
that we will denote by $\overline C_m=\overline C_m^{p,q}$.
By \cite[Proposition 9.1]{Ric} the complex $\overline C_{m-1}\oplus \overline C_{m}$ realizes a derived equivalence between $\Lambda^{p,q}$ and $(\Lambda^{q,p}_m)^{\rm op}$.

Let $p\ge 2$ again. Suppose that we have a two-term complex $$C=\left(P_1^a\xrightarrow{\tiny\begin{pmatrix}w_{1,1}&\cdots&w_{1,a}\\\vdots&\ddots&\vdots\\w_{b,1}&\cdots&w_{b,a}\end{pmatrix}}P_2^b\right)$$ with $w_{i,j}\in \Hom_{\Lambda^{p,q}}(P_1,P_2)=e_2\Lambda^{p,q}e_1.$ We define 
$$C^*:=\left(P_1^b\xrightarrow{\tiny\begin{pmatrix}w_{1,1}&\cdots&w_{b,1}\\\vdots&\ddots&\vdots\\w_{1,a}&\cdots&w_{b,a}\end{pmatrix}}P_2^a\right).$$
It is easy to see that $(C^*)^*=C$ and that $C^*$ is indecomposable, radical or presilting if and only if $C$ satisfies the same condition. This assertion is based on the fact that to verify the presilting condition for $C$ and $C^*$ one has to deal only with identity morphisms on $P_1$ and $P_2$ and morphisms from $P_1$ to $P_2$.
On the other hand, the same assertion for the pretilting property fails. Indeed, $C_0^*\oplus C_1^*=P_1[1]\oplus \omega_{q,p}(P_1)$ is not a tilting complex if $q>0$. In fact, one has $\End_{\Kb_{\Lambda^{p,q}}}(C_0^*)=\End_{\Kb_{\Lambda^{p,q}}}(C_1^*)=\kk$, $\Hom_{\Kb_{\Lambda^{p,q}}}(C_0^*,C_1^*[i])=0$ for all $i\in\mathbb{Z}$, $\Hom_{\Kb_{\Lambda^{p,q}}}(C_1^*,C_0^*)=\kk^p$ and $\Hom_{\Kb_{\Lambda^{p,q}}}(C_1^*,C_0^*[-1])=\kk^q$, and hence $C_0^*\oplus C_1^*$ is related to a derived equivalence between the algebra $\Lambda^{p,q}$ and the {\it graded $(p+q)$-Kronecker} algebra $\delta_{p,q}$ with $p$ arrows of degree $0$ and $q$ arrows of degree $-1$ (see \cite{LY}). On the other hand, for $m\ge 1$, the radical representative of $\omega_{q,p}^{-1}(C_m^*)$ coincides as a graded $\Lambda^{q,p}$-module with $\overline C_{m-2}^{q,p}$. One can see this, for example, representing
$C_m^*$ as a cocone of a morphism from the complex $(P_1^p\rightarrow P_2)^{a_{m-1}}$ to $P_1^{a_{m-2}}[1]$ and noting that $\omega_{q,p}^{-1}(C_m^*)$ is a cocone of a morphism from $P_1^{a_{m-1}}$ to $P_2^{a_{m-2}}\rightarrow P_1^{pa_{m-2}}$.
Since an indecomposable two-term presilting complex is determined by its $g$-vector, we have $C_m^*=\omega_{q,p}\Hom_{\Lambda^{p,q}}(C_{m-2},\Lambda^{p,q})$ for $m\ge 1$, and, in particular, $(C_{m-1}\oplus C_m)^*$ realizes a derived equivalence between $\Lambda^{p,q}$ and $(\Lambda^{p,q}_{m-2})^{\rm op}$ for $m\ge 2$.
Note that $C_m^*$ is the complex
$$
\bA_{m}^*\otimes P_1\xrightarrow{\sum\limits_{k=1}^p\kappa_k^*\otimes\alpha_k}  \bA_{m-1}^*\otimes P_2.
$$

Finally, suppose again that $q\ge 2$. Note that the complex
$$
\bA_{m-1}\otimes P_2\xrightarrow{\sum\limits_{k=1}^q\kappa_k\otimes\beta_k}  \bA_{m}\otimes P_1
$$
equals $\Hom_{\Lambda^{q,p}}\big((C_m^{q,p})^*,\Lambda^{q,p}\big)[1]$. We will denote this complex by $\overline C_m^*=(\overline C_m^{p,q})^*$. By \cite[Proposition 9.1]{Ric} the complex $\overline C_{m-1}^*\oplus \overline C_m^*$ realizes a derived equivalence between $\Lambda^{p,q}$ and $\Lambda^{p,q}_{m-2}$ for $m\ge 2$ while $\overline C_{0}^*\oplus \overline C_1^*\cong P_1\oplus \omega_{p,q}^{-1}(P_1[1])$ is a two-term silting complex related to a derived equivalence between the algebra $\Lambda^{p,q}$ and the graded $(p+q)$-Kronecker algebra $\delta_{q,p}$ with $q$ arrows of degree $0$ and $p$ arrows of degree $-1$.

Let us define the set
$$
\stilt_{p,q}=\begin{cases}\{ C_{m-1}^{p,q}\oplus C_m^{p,q}\}_{m\ge 0}\cup\{(C_{m-1}^{p,q}\oplus C_m^{p,q})^*\}_{m\ge 1},&\mbox{ if $p\ge 2$},\\
\{\Lambda,P_2\oplus(P_1\rightarrow P_2),(P_1\rightarrow P_2)\oplus P_1[1]\},&\mbox{ if $p=1$},\\
\{\Lambda,P_2\oplus P_1[1]\},&\mbox{ if $p=0$}.
\end{cases}
$$
We also will denote the set $\{\Hom_{\Lambda^{q,p}}(C,\Lambda^{q,p})\mid C[1]\in\stilt_{q,p}\}$ by $\overline\stilt_{q,p}$.
Now we can formulate and prove our second main result.

\begin{theorem} Let $p>0$ and $q$ be integers. Then the set of basic silting complexes is $\stilt_{p,q}\cup\,\overline\stilt_{q,p}$. All of these complexes are tilting except for the complexes $P_1[1]\oplus \omega_{q,p}(P_1)$ and $P_1\oplus \omega_{p,q}^{-1}(P_1[1])$. In particular, the derived equivalence class of $\Lambda^{p,q}$ is the set
$$
[\Lambda^{p,q}]=\begin{cases}
\{\Lambda^{p,q}\},&\mbox{ if $q=0$ or $p=q=1$},\\
\{\Lambda^{p,q}\}\cup\{\Lambda^{q,p}_m,(\Lambda^{q,p}_m)^{\rm op}\}_{m\ge 0},&\mbox{ if $q>p=1$},\\
\{\Lambda^{q,p}\}\cup\{\Lambda^{p,q}_m,(\Lambda^{p,q}_m)^{\rm op}\}_{m\ge 0},&\mbox{ if $p>q=1$},\\
\{\Lambda^{p,q}\}\cup\{\Lambda^{p,q}_m,(\Lambda^{p,q}_m)^{\rm op}\}_{m\ge 1},&\mbox{ if $p=q\ge 2$},\\
\{\Lambda^{p,q}_m,(\Lambda^{p,q}_m)^{\rm op},\Lambda^{q,p}_m,(\Lambda^{q,p}_m)^{\rm op}\}_{m\ge 0},&\mbox{ if $p,q\ge 2$ and $p\not=q$}.\\
\end{cases}
$$
\end{theorem}
\begin{proof} The second assertion follows from the first one and the discussion above. Here we consider the case $p,q\ge 2$. The other cases are easier and can be considered in the same way.

We have also already proved that the set $\stilt_{p,q}\cup\,\overline\stilt_{q,p}$
consists of two-term silting complexes. Moreover, the pairs of $g$-vectors of complexes from $\stilt_{p,q}$ are
\begin{multline*}
\big((1,0),(0,1)\big),\ \big((-a_{p,m-1},a_{p,m}),(-a_{p,m},a_{p,m+1})\big)\,\,(m\ge 0)\mbox{ and }\\
\big((-a_{p,m},a_{p,m-1}),(-a_{p,m+1},a_{p,m})\big)\,\,(m\ge 0),
\end{multline*}
while for the set $\overline\stilt_{q,p}$ we get the pairs
\begin{multline*}
\big((-1,0),(0,-1)\big),\ \big((a_{q,m-1},-a_{q,m}),(a_{q,m},-a_{q,m+1})\big)\,\,(m\ge 0)\mbox{ and }\\
\big((a_{q,m},-a_{q,m-1}),(a_{q,m+1},-a_{q,m})\big)\,\,(m\ge 0).
\end{multline*}
As we said before, the limits of $\frac{a_{p,m+1}}{a_{p,m}}$ and $\frac{a_{q,m+1}}{a_{q,m}}$ are $\frac{p+\sqrt{p^2-4}}{2}$ and $\frac{q+\sqrt{q^2-4}}{2}$ while $m$ tends to infinity. This means that the arcs $L_C$ with $C\in\stilt_{p,q}\cup\,\overline\stilt_{q,p}$ cover the whole unit circle except two arcs; namely the minimal arc cut by the rays $\left(-1,\frac{p+\sqrt{p^2-4}}{2}\right)$ and $\left(-1,\frac{p-\sqrt{p^2-4}}{2}\right)$ and the minimal arc cut by the rays $\left(\frac{q+\sqrt{q^2-4}}{2},-1\right)$ and $\left(\frac{q-\sqrt{q^2-4}}{2},-1\right)$ (see the picture in the end of the paper).

Let us recall that any indecomposable two-term silting complex $X$ such that $O_X\in \bigcup\limits_{C\in \stilt_{p,q}\cup\,\overline\stilt_{q,p}}L_C$ belongs to $\stilt_{p,q}\cup\,\overline\stilt_{q,p}$.
In particular, if there exists some indecomposable two-term silting complex $X$ that does not belong to $\stilt_{p,q}\cup\,\overline\stilt_{q,p}$, then its $g$-vector  belongs to one of the gray sections in the picture below. This means that either $X=(P_1^a\rightarrow P_2^b)$ with $\frac{p-\sqrt{p^2-4}}{2}\le \frac{b}{a} \le \frac{p+\sqrt{p^2-4}}{2}$ or  $X=P_2^a\rightarrow P_1^b$ with $\frac{q-\sqrt{q^2-4}}{2}\le \frac{b}{a} \le \frac{q+\sqrt{q^2-4}}{2}$. One has $a^2+b^2-pab\le 0$ in the first case and $a^2+b^2-qab\le 0$ in the second case.
Let us consider the case where $X=P_1^a\rightarrow P_2^b$ with $pab\ge a^2+b^2$. One has $\dim_\kk\Hom_{\Cb_{\Lambda^{p,q}}}(X,X[1])=pab$ while morphisms from $P_1^a$ to $P_1^a$ and from $P_2^b$ to  $P_2^b$ with components in the Jacobson radical of $\Lambda^{p,q}$ and the identity morphism on $X$ give zero morphisms from $X$ to $X[1]$. Thus, the dimension of the space of null homotopic maps from $X$ to $X[1]$ is not greater than $a^2+b^2-1$. Thus, $X$ can not be silting. The case $X=P_2^a\rightarrow P_1^b$ can be considered in the same manner.
\end{proof}

The picture below illustrates the discussion in this section.
The rays are denoted by indecomposable two-term presilting complexes and correspond to their $g$-vectors.
 Each pair of consecutive rays gives a silting complex and the corresponding derived equivalent algebra is written in the cone bounded by these rays.
Cones corresponding to two-term silting complexes that are not tilting are  marked by $\delta_{p,q}$ and $\delta_{q,p}$. 
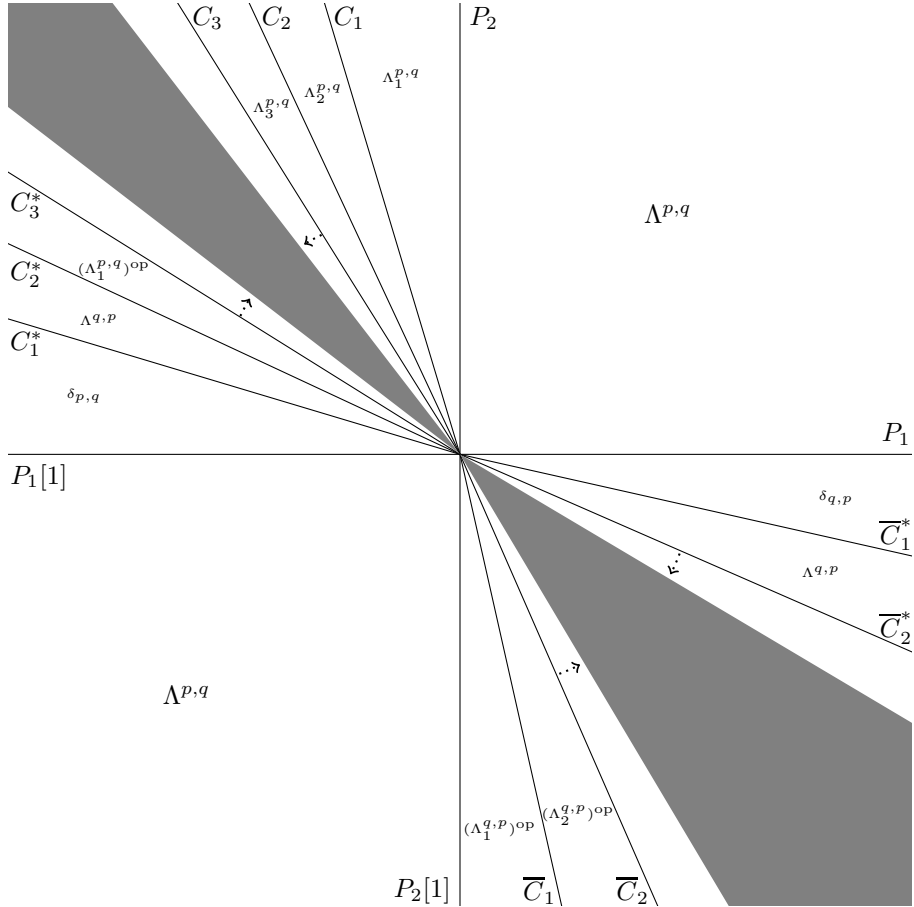
\begin{center}
\begin{tikzpicture}[scale=0.4]
\clip (-15,-15) rectangle (15,15);
\draw (-20,0) -- (20,0);
\draw (0,-20) -- (0,20);
\draw (0,0) -- (-20,6);
\draw (0,0) -- (-30,14);
\draw (0,0) --  (-80, 50);
\draw (0,0) -- (-6,20);
\draw (0,0) -- (-14,30);
\draw (0,0) --  (-50,80);
\fill[gray] (0,0) -- (-26,20) -- (-20,26) -- (0,0);
\draw (0,0) -- (40,-9);
\draw (0,0) --  (160, -70);
\draw (0,0) -- (9,-40);
\draw (0,0) --  (70, -160);
\fill[gray] (0,0) -- (37,-22) -- (22,-37) -- (0,0);
\coordinate [label=left:$\Lambda^{p,q}$] () at (-8,-8);
\coordinate [label=left:$\Lambda^{p,q}$] () at (8,8);
\coordinate [label=above:$P_1$] () at (14.5,0);
\coordinate [label=above:$\overline C_1^*$] () at (14.5,-3.5);
\coordinate [label=above:$\overline C_2^*$] () at (14.5,-6.6);
\coordinate [label=below:$P_1[1{]}$] () at (-14,0);
\coordinate [label=below:$C_1^*$] () at (-14.4,4.4);
\coordinate [label=below:$C_2^*$] () at (-14.4,6.7);
\coordinate [label=below:$C_3^*$] () at (-14.4,9);
\coordinate [label=right:$P_2$] () at (0,14.5);
\coordinate [label=right:$C_1$] () at (-4.5,14.5);
\coordinate [label=right:$C_2$] () at (-6.9,14.5);
\coordinate [label=right:$C_3$] () at (-9.2,14.5);
\coordinate [label=left:$P_2[1{]}$] () at (0,-14.5);
\coordinate [label=left:$\overline C_1$] () at (3.5,-14.5);
\coordinate [label=left:$\overline C_2$] () at (6.6,-14.5);
\coordinate [label=above:$\scriptscriptstyle\delta_{p,q}$] () at (-12.5,1.4);
\coordinate [label=above:$\scriptscriptstyle\Lambda^{q,p}$] () at (-12,4);
\coordinate [label=above:$\scriptscriptstyle(\Lambda^{p,q}_1)^{\rm op}$] () at (-11.5,5.5);
\coordinate [label=left:$\scriptscriptstyle\Lambda^{p,q}_1$] () at (-1,12.5);
\coordinate [label=left:$\scriptscriptstyle\Lambda^{p,q}_2$] () at (-3.6,12);
\coordinate [label=left:$\scriptscriptstyle \Lambda^{p,q}_3$] () at (-5.3,11.5);
\coordinate [label=right:$\scriptscriptstyle(\Lambda^{q,p}_1)^{\rm op}$] () at (-0.2,-12.5);
\coordinate [label=right:$\scriptscriptstyle(\Lambda^{q,p}_2)^{\rm op}$] () at (2.4,-12);
\coordinate [label=above:$\scriptscriptstyle \delta_{q,p}$] () at (12.5,-2);
\coordinate [label=above:$\scriptscriptstyle \Lambda^{q,p}$] () at (12,-4.4);
\draw[dotted, thick, ->] (3.3,-7.3) arc (-70: -62: 5.4);
\draw[dotted, thick, ->] (7.3,-3.3) arc (-20: -28: 5.4);
\draw[dotted, thick, ->] (-4.6,7.3) arc (115: 122: 5.4);
\draw[dotted, thick, ->] (-7.3,4.6) arc (155: 148: 5.4);
\end{tikzpicture}
\end{center}

{\bf Acknowledgements.} The author is grateful to Alexandra Zvonareva for productive discussions and helpful advises. The work was supported by  RFBR according to the research project 18-31-20004 and by the President's Program ``Support of Young Russian Scientists'' (project number MK-2262.2019.1).

\end{document}